\documentclass[12pt,a4paper]{amsart}
\usepackage{amssymb}
\usepackage{graphicx}
\usepackage[matrix,arrow]{xy}

\newcommand{\image}{{\rm{Im}}}
\newcommand{\End}{{\rm{End}}}
\newcommand{\Hom}{{\rm{Hom}}}
\newcommand{\Q}{\ensuremath{\mathbb{Q}}}
\newcommand{\Z}{\ensuremath{\mathbb{Z}}}

\newcommand{\Zpl}{\ensuremath{\Z_{(p)}}}

\newcommand{\bpn}{\ensuremath{BP\langle n\rangle}}
\newcommand{\bpm}{\ensuremath{BP\langle m\rangle}}

\theoremstyle{plain}
\numberwithin{equation}{section}
\newtheorem{theorem}[equation]{Theorem}
\newtheorem{proposition}[equation]{Proposition}

\newtheorem{lemma}[equation]{Lemma}

\newtheorem{remark}[equation]{Remark}

\theoremstyle{definition}
\newtheorem{definition}[equation]{Definition}

\begin{document}

\title{Central cohomology operations and $K$-theory}
\author{Imma G\'alvez-Carrillo
 \and Sarah Whitehouse}                             
\address{Departament de Matem\`{a}tica Aplicada III,
Escola D'Enginyeria de Terrassa, Universitat Polit\`{e}cnica de Catalunya,
C/ Colom 1, 08222 Terrassa, Spain.}
\email{m.immaculada.galvez@upc.edu} 
\address{School of Mathematics and Statistics, University of Sheffield, Sheffield S3 7RH, UK.} 
\email{s.whitehouse@sheffield.ac.uk}

\date{$13^{\text{th}}$ April 2012}

\begin{abstract}
For stable degree zero operations, and also for additive unstable 
operations of bidegree $(0,0)$, it is known that the centre of the ring of operations
for complex cobordism is isomorphic to the corresponding ring of connective
complex $K$-theory operations.
Similarly, the centre of the ring of $BP$ operations
is the corresponding ring for the Adams summand of $p$-local connective complex $K$-theory.
Here we show that, in the additive unstable context, this result holds with $BP$ replaced by
$\bpn$ for any $n$. Thus, for all chromatic heights, the only central operations are
those coming from $K$-theory.

\end{abstract}


\maketitle

\section{Introduction}
\label{intro}

We study cohomology operations for various cohomology theories related to complex cobordism
and show that, in a suitable context, the central cohomology operations are precisely those
coming from complex $K$-theory. Specifically,
we consider the ring of additive unstable bidegree $(0,0)$ operations for the Adams summand of $p$-local
complex $K$-theory and we show that this ring maps via an injective ring 
homomorphism to the corresponding ring of operations for the
theory $\bpn$, for all $n\geq 1$. The image of this map is the centre of the 
target ring. 

Previously results of this type had been established with target $BP$
(which may be regarded as the $n=\infty$ case) in both the stable and additive unstable
contexts; see~\cite{GW} and~\cite{strongw1}.

The $\bpn$ result that we give here is quite a simple consequence of combining certain unstable
$BP$ splittings due to Wilson~\cite{Wilson75} with the results of~\cite{strongw1}.
Nonetheless we think it is interesting since it shows that the central operations
are precisely those arising from $K$-theory at every chromatic height. 
\bigskip

Let $p$ be an odd prime and let $BP$ be the $p$-local Brown-Peterson spectrum,
a summand of the $p$-local complex bordism spectrum $MU_{(p)}$. 
For each $n\geq0$, 
there is a connective commutative ring spectrum $\bpn$ with coefficient groups
    $$
    \bpn_*=\Zpl[v_1, v_2, \dots, v_n]=BP_*/(v_{n+1}, v_{n+2}, \dots )=BP_*/J_n.
    $$
Here $BP_*=\Zpl[v_1, v_2, \dots]$, where the $v_i$s are Hazewinkel's generators, with
$v_i$ in degree $2(p^i-1)$ 
and $J_n=(v_{n+1}, v_{n+2}, \dots )$.
These theories were introduced by Wilson in~\cite{Wilson75} and further studied
by Johnson and Wilson in~\cite{JohnsonWilson73}.
They fit into a tower of $BP$-module spectra:
    $$
    \xymatrix{
    BP\ar[r]&\dots\ar[r]&\bpn\ar[r] &BP\langle n-1 \rangle\ar[r] &\dots\\
    }
    $$

In particular, $BP\langle 0\rangle=H\Zpl$ and
$BP\langle 1\rangle=g$, the Adams summand of connective $p$-local complex $K$-theory.
\medskip

Recall that for a cohomology theory $E$, the
bidegree $(0,0)$ unstable operations are given by $E^0(\underline{E}_0)$,
where
$\underline{E}_0$ denotes the $0$-th space of the $\Omega$-spectrum representing
the cohomology theory $E$.
Inside here we have
$PE^0(\underline{E}_0)$, the additive bidegree $(0,0)$ unstable operations,
which we will denote by $\mathcal{A}(E)$. This is a ring, with multiplication given by
composition of operations.

Using unstable $BP$ splittings due to Wilson, we will define an injective ring
homomorphism $\hat{\iota}_n:\mathcal{A}(g)\to \mathcal{A}(\bpn)$. Our main result,
Theorem~\ref{main}, is that the image of $\hat{\iota}_n$ is the centre 
of $\mathcal{A}(\bpn)$.

The situation is analogous to that of matrix rings, where the diagonal matrices
form the centre of the $n\times n$ matrices for all $n$. Indeed, we will see that
all operations considered are determined by the matrices giving their actions
on homotopy groups. Of course, not all matrices arise as actions
of operations; there are complicated constraints. Essentially, what we show is
that, 
\begin{enumerate}
\item at every height $n$, enough matrices arise so that central operations are forced to act diagonally (in a suitable sense), and
\item the constraints on the diagonal operations which can occur are the same for all $n$.
\end{enumerate}
\medskip

This paper is organized as follows. In Section~\ref{sec:splittings} we explain some of Wilson's results on unstable
$BP$ splittings and deduce faithfulness of the actions of additive $\bpn$ operations of bidegree $(0,0)$ on homotopy groups. In the next section
we recall some results on additive operations for the Adams summand $g$ of connective $p$-local complex $K$-theory.
We also
define our map of operations $\mathcal{A}(g)\to \mathcal{A}(\bpn)$ and give its basic properties. In Section~\ref{sec:diag} we
define and study diagonal operations. Section~\ref{sec:cong} contains the proof of our main
result, Theorem~\ref{main}, that the image of the
map coincides with the centre of the target ring. 

\section*{Acknowledgements}
The first author gratefully acknowledges support
from the grants
MTM2010-15831 ``Homotop\'{\i}a de
orden superior en \'algebra y geometr\'{\i}a" and
MTM2010-20692 ``An\'alisis local en grupos y espacios topol\'ogicos" from
the FEDER-Ministerio de Econom\'{\i}a y Competitividad, Spanish Government, and
grant 2009SGR-1092 ``Grup de Topologia Algebraica de Barcelona" from
AGAUR, Generalitat de Catalunya. The second author gratefully 
acknowledges support from the former two grants
 for visits to Barcelona.

\section{Unstable splittings}
\label{sec:splittings}
In this section we begin by recalling some results on unstable $BP$ splittings. These results are due to
Wilson~\cite{Wilson75}; we use~\cite{bjw} as our main reference. We then deduce some straightforward consequences
for operations. 

As usual, let $\underline{E}_k$ denote the $k$-th space of the $\Omega$-spectrum representing
the cohomology theory $E$. For $n\geq 0$, write $\pi_n:\underline{BP}_0\to\underline{\bpn}_0$ for the
map coming from the map of $BP$-module spectra $BP\to\bpn$. The induced map on homotopy,
$(\pi_n)_*: BP_*\to \bpn_*=BP_*/J_n$, is the canonical projection.

The following lemma is the special case of~\cite[Lemma 22.1]{bjw} for zero spaces.

\begin{lemma}\cite[Lemma 22.1]{bjw}
For all $n\geq 0$, there is an $H$-space
splitting $\theta_n: \underline{\bpn}_0\to \underline{BP}_0$ of $\pi_n$. Let
$e_n=\theta_n\pi_n$ denote the corresponding additive idempotent $BP$-operation;
the choices can be made compatibly so that $e_ne_m=e_me_n=e_m$ for $m<n$.\qed
\end{lemma}

These splittings immediately allow us to compare operations. 

\begin{lemma}\label{splitting}
We have maps 
    $$
    i_n: \mathcal{A}(\bpn)\rightleftarrows \mathcal{A}(BP) :p_n
    $$
such that
\begin{enumerate}
\item $i_np_n: \mathcal{A}(BP)\to \mathcal{A}(BP)$ is given by $[f]\mapsto
[e_nfe_n]$;
\item $p_n$ splits $i_n$ (so $i_n$ is injective and $p_n$ is surjective);
\item $i_n$ is a non-unital ring homomorphism;
\item $p_n$ is an additive group homomorphism.
\end{enumerate} 
\end{lemma}

\begin{proof}
We have the maps
    \begin{align*}
    [\theta_n\circ-\circ\pi_n]:\bpn^0(\underline{\bpn}_0)&\to BP^0(\underline{BP}_0)\\
    [f]&\mapsto [\theta_nf\pi_n]
    \end{align*}
and 
    \begin{align*}
    [\pi_n\circ-\circ\theta_n]:BP^0(\underline{BP}_0)&\to \bpn^0(\underline{\bpn}_0)\\
    [f]&\mapsto [\pi_nf\theta_n].
    \end{align*}
Since $\pi_n$ and $\theta_n$ are $H$-space maps, these maps restrict to maps
on the additive operations, which we denote by $i_n$ and $p_n$ respectively.

The first property follows from
$\theta_n\pi_n=e_n$. The remaining properties are easy to check using that $\pi_n\theta_n\simeq id$ and that $\theta_n$
is a map of $H$-spaces.
\end{proof}

\begin{remark}\rm
It follows that we may $\mathcal{A}(\bpn)$ identify with the subring $e_n \mathcal{A}(BP)e_n$
of $\mathcal{A}(BP)$.
\end{remark}

In the following lemma, we will record some information about actions in homotopy.
Note that we can regard $\bpn_*=\Zpl[v_1, \dots, v_n]$ as both a subring and
a quotient ring of $BP_*$. We will abuse notation by writing the inclusion silently
and we use $[\,\,]$ to denote classes in $\bpn_*=BP_*/J_n$.

\begin{lemma}
\label{actions}
\begin{enumerate}
\item For $x\in \bpn_*$, $(\theta_n)_*(x)\equiv x\mod J_n$. In particular,
$(\theta_n)_*$ is the identity in degrees less than $2(p^{n+1}-1)$.
\item Let $\phi\in\mathcal{A}(\bpn)$. Then
\begin{enumerate}
\item for $y\in BP_*$,
$(i_n\phi)_*(y)\equiv \phi_*([y])\mod J_n$, and
\item for $y\in J_n$, $(i_n\phi)_*(y)=0$.
\end{enumerate}
\item Let $\varphi\in\mathcal{A}(BP)$ and suppose that $\varphi_*(J_n)\subseteq J_n$.
Then, for $z\in\bpn_*$, $(p_n\varphi)_*(z)=[\varphi_*(z)]$.
\end{enumerate}
\end{lemma}

\begin{proof}
Part (1) is immediate from $\pi_n\theta_n\simeq id$. Then, for part (2), for $y\in BP_*$,
    $$
    (i_n\phi)_*(y)  =(\theta_n)_*\phi_*(\pi_n)_*(y)
                    =(\theta_n)_*\phi_*([y])
                    \equiv \phi_*([y])\mod J_n,
    $$
and for $y\in J_n$,
    $$
    (i_n\phi)_*(y)  =(\theta_n)_*\phi_*(\pi_n)_*(y)
                    =(\theta_n)_*\phi_*(0) =0.
    $$
Finally, for part (3), we have,
    \begin{align*}
    (p_n\varphi)_*(z)
        &=(\pi_n)_*\varphi_*(\theta_n)_*(z)\\
        &=(\pi_n)_*\varphi_*(z+w)\qquad\text{for some $w\in J_n$}\\
        &=(\pi_n)_*(\varphi_*(z)+\varphi_*(w)) \qquad\text{since $\varphi$ is additive}\\
        &=[\varphi_*(z)]\qquad\text{since, by hypothesis, $\varphi_*(w)\in J_n$.}\qedhere
    \end{align*}
\end{proof}

\begin{remark}\rm
It is worth noting that $(\theta_n)_*$ is \emph{not} the obvious splitting on homotopy
groups with image $\Zpl[v_1, \dots, v_n]$ (and it is not a ring homomorphism). 
See~\cite[p817]{bjw} for an example.
\end{remark}
\medskip

Another important consequence of the splitting is that the action of the additive
 $\bpn$ operations of
 bidegree $(0,0)$
 on homotopy
groups is faithful. As we will see, the splitting allows us to deduce this from
the corresponding result for $BP$, which was proved in~\cite[Proposition 1]{strongw1}.
(Key ingredients for the $BP$ case are that $BP$-theory has good duality and that everything is torsion-free.)   
\smallskip

Given an unstable $E$-operation $\theta\in  E^0(\underline{E}_0)\cong [\underline{E}_0, \underline{E}_0]$, we may
consider the induced homomorphism of graded abelian groups
$\theta_*: \pi_*( \underline{E}_0)\to \pi_*( \underline{E}_0)$ given by the
action of $\theta$ on homotopy groups. For a graded abelian group $M$, we write
$\End(M)$ for the ring of homomorphisms of graded abelian groups from $M$ to itself.

Sending an operation to its action
on homotopy groups gives a map
    \begin{align*}
    E^0(\underline{E}_0) &\to \End(\pi_*( \underline{E}_0))\\
    \phi &\mapsto \phi_*.
    \end{align*}
The restriction of this map to the additive $E$-operations $\mathcal{A}(E)$ 
is a ring homomorphism
and we denote this
by $\beta_E$:
    \begin{align*}
    \beta_E: \mathcal{A}(E) &\to \End(\pi_*( \underline{E}_0))\\
    \phi &\mapsto \phi_*.
    \end{align*}

\begin{proposition}\label{faithfulaction}
For all $n\geq 0$, the ring homomorphism
    $$
    \beta_{\bpn}:\mathcal{A}(\bpn) \to 
    \End(\pi_*( \underline{\bpn}_0))
    $$ 
is injective.
\end{proposition}

\begin{proof}
Let $\phi\in \mathcal{A}(\bpn)$ and suppose that $\beta_{\bpn}(\phi)=\phi_*=0$.
Then 
    $$
    \beta_{BP}(i_n(\phi))=(i_n(\phi))_*=(\theta_n\phi\pi_n)_*=(\theta_n)_*\phi_*(\pi_n)_*=0.
    $$
But $\beta_{BP}$ is injective (see~\cite[Proposition 1]{strongw1}) and so is $i_n$, so
$\phi=0$.
\end{proof}

\section{The comparison map}

In this section we begin with some reminders about the additive operations for the Adams summand $g$
of $p$-local connective complex $K$-theory and we recall the main result of~\cite{strongw1}. We then go on to define
the main map to be studied in this paper, $\hat{\iota}_n:\mathcal{A}(g)\to \mathcal{A}(\bpn)$,
and we discuss its basic properties.
\smallskip

A description of the ring of additive operations $\mathcal{A}(g)$ for the Adams summand can be
deduced from the corresponding result for integral complex $K$-theory (see~\cite[Lecture 4]{adams-sln99}).
Another description can be found in~\cite{strongw2}: Theorems 3.3 and 4.2 of~\cite{strongw2} together give 
a topological basis for this ring, where the basis elements are
certain polynomials in the Adams operations $\Psi^0$, $\Psi^p$ and $\Psi^q$ (where $q$ is primitive
modulo $p^2$ and thus a topological generator for the $p$-adic units). The precise details of the
description are not needed here; what is important to note is that all operations can be
described in terms of Adams operations.
\smallskip

The main result of~\cite{strongw1} (in the split case) is the following.
\smallskip

\begin{theorem}~\cite[Theorem 19]{strongw1}\label{theorem:BP}
There is an injective ring homomorphism $\hat{\iota}:\mathcal{A}(g)\to \mathcal{A}(BP)$
such that the image is precisely the centre of the ring $\mathcal{A}(BP)$.\qed
\end{theorem}

It is worth noting that $\hat{\iota}$
is different from the ring homomorphism 
 $i_1: \mathcal{A}(BP\langle 1\rangle)=\mathcal{A}(g)\to \mathcal{A}(BP)$
 provided by Lemma~\ref{splitting}.
 Indeed,
$\hat{\iota}$ sends the identity operation of $g$ to the identity operation of $BP$, 
whereas $i_1$ does not. More generally, it is
instructive to consider the effects of these two maps on Adams operations: $\hat{\iota}$
 takes the Adams operation $\Psi_g^k$ of the Adams summand to the
corresponding Adams operation $\Psi_{BP}^k$ for $BP$; this operation acts as
multiplication by $k^{(p-1)n}$ on each element of the group $\pi_{2(p-1)n}(BP)$. On the other hand,
$i_1$ sends $\Psi_g^k$ to an operation which acts as zero on the ideal $J_1$.

The two maps share the property of being split by $p_1$.

\begin{lemma}
The map $p_1$ also splits $\hat{\iota}$.
\end{lemma}

\begin{proof}
All elements of the topological ring $\mathcal{A}(g)$ can be explicitly expressed 
as certain (infinite) linear combinations of Adams
operations; see~\cite[Proposition 18]{strongw1}. By Lemma~\ref{splitting}, $p_1$ is additive,
and it is straightforward to see that it is continuous with respect to the profinite filtrations on
the rings of operations. Thus it is enough to check that $p_1\hat{\iota}(\Psi_g^k)=\Psi_g^k$
for all $k\in\Zpl$. 
Now $\hat{\iota}(\Psi_g^k)=\Psi_{BP}^k$ and this $BP$ Adams operation
acts as multiplication by $k^{(p-1)r}$ on $BP_{2(p-1)r}$ (and so, in particular, preserves $J_1$).
Then by part (3) of Lemma~\ref{actions},
 $p_1\hat{\iota}(\Psi_g^k)$ acts
as multiplication by $k^{(p-1)r}$ on $\pi_{2(p-1)r}(g)=\Zpl\langle v_1^r\rangle$. But,
by~\cite[Proposition 1]{strongw1},
this completely characterizes $\Psi_g^k$.
\end{proof}

The main map we will consider comes from composing the map $\hat{\iota}: \mathcal{A}(g)\to \mathcal{A}(BP)$ 
of Theorem~\ref{theorem:BP} with the map $p_n: \mathcal{A}(BP)\to \mathcal{A}(\bpn)$ of Lemma~\ref{splitting}.

\begin{definition}
Define $\hat{\iota}_n=p_n\hat{\iota}:\mathcal{A}(g)\to \mathcal{A}(\bpn)$.
\end{definition}

Note that this gives us our map of operations without explicitly mentioning
Adams operations for $\bpn$. On the other hand, we can define such Adams operations
as follows.

\begin{definition}
Define unstable Adams operations for $\bpn$ as the images of the corresponding
$BP$ operations:
    $$
    \Psi_{\bpn}^k:=p_n(\Psi_{BP}^k),
    $$
for $k\in\Zpl$.
\end{definition}

Using part (3) of Lemma~\ref{actions}, we see
that this definition gives unstable Adams operations for $\bpn$ with the expected
actions on homotopy (namely, $\Psi_{\bpn}^k(z)=k^{(p-1)r}z$, for $z\in\bpn_{2(p-1)r}$).
\smallskip

Since $\hat{\iota}(\Psi_{g}^k)=\Psi_{BP}^k$, it follows from this definition of the
Adams operations for $\bpn$ and the description of $\mathcal{A}(g)$ in terms of Adams operations, that
the map $\hat{\iota}_n$ is determined by mapping $g$ Adams operations to the
corresponding $\bpn$ Adams operations and extending to (suitable infinite) linear
combinations.
\medskip

Our main result will be that the analogue of Theorem~\ref{theorem:BP} holds for
$\hat{\iota}_n: \mathcal{A}(g)\to \mathcal{A}(\bpn)$.
 We begin with some basic properties of $\hat{\iota}_n$;
in particular, it is a ring homomorphism (even though $p_n$ is not).

\begin{proposition}\label{mapprops}
For all $n\geq 1$, the map $\hat{\iota}_n:\mathcal{A}(g)\to \mathcal{A}(\bpn)$ is an
injective unital ring homomorphism whose image is contained in the centre
of $\mathcal{A}(\bpn)$.
\end{proposition}

\begin{proof}
First we check that $\hat{\iota}_n$ is a ring homomorphism.
For $a,b \in \mathcal{A}(g)$, we have
    \begin{align*}
    i_n(\hat{\iota}_n(a)\hat{\iota}_n(b))&=i_n(\hat{\iota}_n(a))i_n(\hat{\iota}_n(b)) &\text{$i_n$ ring homomorphism}\\
    &=i_np_n\hat{\iota}(a)i_np_n\hat{\iota}(b)                                        &\text{definition of $\hat{\iota}_n$}\\
    &=e_n\hat{\iota}(a)e_n^2\hat{\iota}(b)e_n                                              &\text{by Lemma~\ref{splitting}}\\
    &=e_n^3\hat{\iota}(a)\hat{\iota}(b) e_n                                                &\text{image of $\hat{\iota}$ central}\\
    &=e_n\hat{\iota}(a)\hat{\iota}(b)  e_n                                             &\text{$e_n$ idempotent}\\
    &=e_n\hat{\iota}(ab)e_n                                                           &\text{$\hat{\iota}$ ring homomorphism}\\
    &=i_np_n\hat{\iota}(ab)                                                           &\text{by Lemma~\ref{splitting}}\\
    &=i_n\hat{\iota}_n(ab)                                                            &\text{definition of $\hat{\iota}_n$}.\\
    \end{align*}
But $i_n$ is injective, so $\hat{\iota}_n(a)\hat{\iota}_n(b)=\hat{\iota}_n(ab)$.

Similarly, we find $i_n(\hat{\iota}_n(1))=e_n=i_n(1)$, so $\hat{\iota}_n(1)=1$ and $\hat{\iota}_n$ is unital.

Next we show injectivity. Let $\phi\in\mathcal{A}(g)$, with $\phi\neq 0$.
By~\cite[Proposition 1]{strongw1}, the action of operations in $\mathcal{A}(g)$ on homotopy
groups is faithful. Thus there is some $r$ such that $\phi$ acts on $\pi_{2(p-1)r}(\underline{g}_0)$ as
multiplication by some non-zero element $\lambda$ of $\Zpl$. But then the action of $\hat{\iota}_n(\phi)$
is given by multiplication by $\lambda\neq 0$ on $\pi_{2(p-1)r}(\underline{\bpn}_0)\neq 0$ and so $\hat{\iota}_n(\phi)\neq 0$.

Finally we need to see that the image is central.
The image consists of certain infinite linear combinations of Adams operations for $\bpn$.
It is clear from the action of $\Psi^k_{\bpn}$
on homotopy that $\beta_{\bpn}(\Psi^k_{\bpn})=(\Psi^k_{\bpn})_*$ commutes with
all elements of $\End(\pi_*( \underline{\bpn}_0))$.
So the same holds for the image under $\beta_{\bpn}$ of (suitable infinite) linear combinations of the Adams operations.
But by Proposition~\ref{faithfulaction},
$\beta_{\bpn}$ is injective, so any element of the image of $\mathcal{A}(g)$ commutes with all elements of $\mathcal{A}(\bpn)$.
\end{proof}

As a consequence of the definitions, we have the following commutative diagram
of abelian groups, for $m\leq n$,
giving the compatibility between the various $\hat{\iota}$ maps.

    $$
    \xymatrix{
    &\mathcal{A}(\bpn)\ar[r]^\cong_{i_n} &e_n\mathcal{A}(BP)e_n\ar[dd]^{e_m\circ-\circ e_m}\\
    \mathcal{A}(g)\ar[ru]^{\hat{\iota}_n}\ar[r]^{\hat{\iota}}
    \ar[dr]_{\hat{\iota}_m}
    &\mathcal{A}(BP)  \ar[u]_{p_n}\ar[d]^{p_m}\\
    &\mathcal{A}(\bpm)\ar[r]^\cong_{i_m} &e_m  \mathcal{A}(BP)e_m\\
    }
    $$

\begin{remark}\rm
It is natural to ask if one can obtain the ring $\mathcal{A}(BP)$ as any kind of limit over the
$\mathcal{A}(\bpn)$, but this does not seem to be the case. 
On the one hand, we can put the $\mathcal{A}(\bpn)$ into a direct system 
of injective ring homomorphisms and produce an injective ring homomorphism
$\displaystyle{\lim_{\rightarrow_n}} \mathcal{A}(\bpn)\to \mathcal{A}(BP)$. However, this is not surjective;
for example the identity operation on $BP$ is not in the image. On the other hand, the maps
in the other direction are not ring homomorphisms, so the inverse limit 
$\displaystyle{\lim_{\leftarrow_n}} \mathcal{A}(\bpn)$ can only be formed in the category of abelian groups.
\end{remark}

\section{Diagonal operations}
\label{sec:diag}
We define unstable diagonal operations for $\bpn$, just as was done for $BP$ in~\cite{strongw1}.

\begin{definition}\label{diagonals}
Write $\mathcal{D}(\bpn)$ for the subring of $\mathcal{A}(\bpn)$ consisting of operations
whose action on each homotopy group $\pi_{2(p-1)r}(\underline{\bpn}_0)$ is multiplication
by an element $\mu_{r}$ of $\Zpl$. We call
elements of $\mathcal{D}(\bpn)$ \emph{unstable diagonal operations}.
\end{definition}

The main result of this section will be that the central operations coincide with the
diagonal operations. One inclusion is easy.

For a ring $R$, we write $Z(R)$ for its centre.

\begin{lemma}\label{diagonalincentral}
We have $\mathcal{D}(\bpn)\subseteq Z(\mathcal{A}(\bpn))$.
\end{lemma}

\begin{proof}
The action on homotopy of $\phi\in\mathcal{D}$ commutes with the action of any operation
in $\mathcal{A}(\bpn)$, so
the inclusion $\mathcal{D}(\bpn)\subseteq Z(\mathcal{A}(\bpn))$ follows
from the faithfulness of the action (Proposition~\ref{faithfulaction}).
\end{proof}

Our proof of the reverse inclusion will amount to finding enough operations
in order to force a central operation to act diagonally. Our strategy will be to 
start from stable $BP$ operations, over which we have better control, and then to
view these as additive unstable operations and project them to $\mathcal{A}(\bpn)$.
\medskip

First we will need some notation for sequences indexing monomials.
We write $v^\alpha$ for the monomial $v_1^{\alpha_1}v_2^{\alpha_2}\dots v_m^{\alpha_m}$,
where $\alpha=(\alpha_1, \alpha_2, \dots, \alpha_m)$ is a sequence
of non-negative integers, with $\alpha_m\neq 0$.
We order such sequences right lexicographically; explicitly
for $\alpha=(\alpha_1, \alpha_2, \dots, \alpha_m)$ and $\beta=(\beta_1, \beta_2, \dots, \beta_n)$,
we have $\alpha<\beta$ if $m<n$ or if $m=n$ and there is some $j$, with
$1\leq j\leq m$, such that $\alpha_k=\beta_k$ for all $k>j$ but $\alpha_j <\beta_j$.

We add sequences placewise: $(\alpha+\beta)_i=\alpha_i+\beta_i$, so that
$v^\alpha v^\beta=v^{\alpha+\beta}$. It is straightforward to check that the ordering
behaves well with respect to the addition: if $\alpha\leq \alpha'$ and $\beta\leq \beta'$
then $\alpha+\beta\leq \alpha'+\beta'$.

The degree of $v^\alpha$ is $2\sum_{i=1}^m\alpha_i (p^i-1)$ and we write this as $|\alpha|$.

\begin{lemma}
Let $\alpha, \beta, \gamma$ denote sequences indexing monomials in the
same degree, $|\alpha|=|\beta|=|\gamma|$. 
\begin{enumerate}
\item There is a stable $BP$ operation $\phi_\beta$ in
$BP^{|\alpha|}(BP)$ whose action $BP_{|\alpha|}\to BP_0=\Zpl$ has the property that
$(\phi_\beta)_*(v^\gamma)=\mu_{\gamma, \beta}$, where 
    \begin{align*}
        \mu_{\beta,\beta}&\neq 0,\\
        \mu_{\gamma,\beta}&=0\quad \text{if $\gamma<\beta$}.
    \end{align*}
\item There is a stable $BP$ operation $\phi_{\alpha, \beta}$ in $BP^0(BP)$
whose action $BP_{|\alpha|}\to BP_{|\alpha|}$ has the property that
$(\phi_{\alpha, \beta})_*(v^\gamma)=\mu_{\gamma, \beta}v^\alpha$, where 
    \begin{align*}
        \mu_{\beta,\beta}&\neq 0,\\
        \mu_{\gamma,\beta}&=0\quad \text{if $\gamma<\beta$}.
    \end{align*}
\end{enumerate}
\end{lemma}

\begin{proof}
The second part follows immediately from the first, by
taking $\phi_{\alpha, \beta}=v^\alpha\phi_\beta$.

For the first part, we recall that $BP$ has good duality and so a
stable operation $\phi$ in $BP^*(BP)$ corresponds to a degree zero $BP_*$-linear functional
$\overline{\phi}:BP_*(BP)\to BP_*$. The action of the operation on coefficient
groups is recovered from the functional by precomposition with the right
unit map, $\eta_R:BP_*\to BP_*(BP)$; that is, $\phi_*=\overline{\phi}\eta_R$. We have $BP_*(BP)=BP_*[t_1, t_2, \dots]$
and so a functional as described above is determined by any choice of its value
on each monomial in the $t$s.

The map $\eta_R$ is of course very complicated, but we will only need to exploit
some basic information about its form. We have
    $$
    \eta_R(v_m)=pt_m+\sum \lambda_\gamma t^\gamma + \sum \mu_{\delta, \delta'}v^\delta t^{\delta'},
    $$
where $\lambda_\gamma, \mu_{\delta, \delta'}\in\Zpl$, $\delta\neq \emptyset$, and $\gamma$ runs
over sequences other than $(0,\dots, 0,1)$ in the degree of $v_m$. (The only
content here is the form of the top term, of course.)
Now $\eta_R$ is a ring map and it follows from the properties of the ordering on monomials described above
that
    $$
    \eta_R(v^\gamma)=\lambda t^\gamma +\sum_{\gamma'<\gamma} \lambda'_{\gamma'} t^{\gamma'} + 
    \sum_{\delta\neq \emptyset} \mu'_{\delta, \delta'}v^\delta t^{\delta'},
    $$
for some $\lambda, \lambda'_{\gamma'}, \mu'_{\delta, \delta'}\in\Zpl$ with
$\lambda\neq 0$.

Now consider the functional $\overline{\phi_\beta}:BP_*(BP)\to BP_*$ which is 
zero on all monomials except $t^\beta$
and sends $t^\beta$ to $1$. By construction the corresponding operation
$\phi_\beta$ has the required property. 
\end{proof}

Now the following lemma follows as a matter of elementary linear algebra.

\noindent Let $E_{\alpha, \beta}$ denote the elementary matrix with a $1$ in the 
$(\alpha, \beta)$ position and zeroes everywhere else.

\begin{lemma}
For all $\alpha, \beta$ with $|\alpha|=|\beta|$, there is some non-zero 
$\overline{\mu}_{\alpha, \beta}\in\Zpl$ and an operation
$\varphi_{\alpha,\beta}$ in $BP^0(BP)$ such that the matrix of its
action on $BP_{|\alpha|}$ is $\overline{\mu}_{\alpha, \beta}E_{\alpha, \beta}$.
\end{lemma}

\begin{proof}
The preceding lemma gives the operation $\phi_{\alpha, \beta}$.
Using the $\Zpl$-basis of monomials in the $v$s, ordered as above, this
operation acts on coefficients in the given degree by the matrix
    $$
    M_{\alpha,\beta}=\sum_{\gamma\geq \beta} \mu_{\gamma, \beta}E_{\alpha, \gamma},
    $$
where $\mu_{\beta, \beta}\neq 0$. 

If we order the elementary matrices by $E_{\beta,\gamma}<E_{\beta',\gamma'}$ if
$\gamma<\gamma'$ or $\gamma=\gamma'$ and $\beta<\beta'$, then the
above shows that the matrix writing the $M_{\alpha,\beta}$ in terms of the
$E_{\alpha,\beta}$ is non-singular lower triangular.
Hence, for some  $\overline{\mu}_{\alpha, \beta}\neq 0$, we can write
$\overline{\mu}_{\alpha, \beta}E_{\alpha, \beta}$ as a $\Zpl$-linear combination
of the $M_{\alpha,\beta}$. We take $\varphi_{\alpha,\beta}$ to be the
corresponding linear combination of the $\phi_{\alpha,\beta}$.
\end{proof}

\begin{theorem}\label{centreisdiag}
We have $Z(\mathcal{A}(\bpn))= \mathcal{D}(\bpn)$.
\end{theorem}

\begin{proof}
We noted the inclusion $\mathcal{D}(\bpn)\subseteq Z(\mathcal{A}(\bpn))$ in
Lemma~\ref{diagonalincentral} above, so it remains to show the reverse
inclusion.

As in the proof of Lemma 11 of~\cite{strongw1}, there is an injection
$BP^0(BP)\hookrightarrow \mathcal{A}(BP)$ 
from the stable degree zero $BP$ operations to the additive unstable bidegree $(0,0)$ operations,
given by sending a stable operation to its zero component (that is, applying $\Omega^\infty$). 
This allows us to view the operation $\varphi_{\alpha,\beta}$ constructed above as an element of $\mathcal{A}(BP)$,
still acting on coefficients in the specified degree as some non-zero multiple of the elementary
matrix $E_{\alpha, \beta}$. 

Next we map these operations to $\mathcal{A}(\bpn)$: consider $p_n(\varphi_{\alpha,\beta})\in 
\mathcal{A}(\bpn)$. We consider the action of this operation on coefficients in degree
$|\alpha|$. Now we can write $BP_{|\alpha|}$ as a direct sum of $\Zpl$-modules $R\oplus J$,
where $R=BP_{|\alpha|}\cap \Zpl[v_1, \dots, v_n]$ and $J=BP_{|\alpha|}\cap J_n$. Notice
that any monomial in the $v$s lying in $R$ is lower in the ordering than any
monomial lying in $J$. So, when we write the action of an operation as a matrix with
respect to the monomial basis, this splits into blocks, according to the decomposition into
$R$ and $J$.

Now let $\alpha, \beta$ index monomials in $R$. Using
$(p_n(\varphi_{\alpha,\beta}))_*=(\pi_n)_*(\varphi_{\alpha,\beta})_*(\theta_n)_*$,
it is easy to check the action
of $p_n(\varphi_{\alpha,\beta})$ is given by 
$\overline{\mu}_{\alpha, \beta}E_{\alpha, \beta}$ on $\bpn_{|\alpha|}$.

So now suppose we have a central operation $\phi\in\mathcal{A}(\bpn)$.
Since it commutes with each operation $\varphi_{\alpha,\beta}$, its action
on $\bpn_{|\alpha|}$ commutes with the action of some non-zero multiple
of each elementary matrix. Hence the matrix of its action in this degree is diagonal with
all diagonal entries equal. That is $\phi\in \mathcal{D}(\bpn)$.
\end{proof}

\section{Congruences}
\label{sec:cong}

Let $S_g$ be the subring of the
infinite direct product $\prod_{i=0}^\infty \Zpl$ consisting of sequences
$(\mu_i)_{i\geq 0}$ satisfying the system of congruences which characterizes
the action on coefficient groups of an element of $\mathcal{A}(g)$.

The congruences can be described as follows; for further details see~\cite[Section 4]{strongw1}.
Let $G$ denote the periodic Adams summand and let
$\{\hat{f}_n\,|\, n\geq 0\}$ be a $\Zpl$-basis for $QG_0(\underline{G}_0)$, where
$Q$ denotes the indecomposable quotient for the $\star$-product. These basis
elements can be written as rational polynomials in the variable $\hat{w}=\hat{u}^{-1}\hat{v}$, 
where $G_*=\Zpl[\hat{u}^{\pm 1}]$ and $\hat{v}=\eta_R(\hat{u})$.
The $n$-th congruence is the condition that the rational linear combination of the $\mu_i$ obtained from 
$\hat{f}_n$ by sending $\hat{w}^i$ to $\mu_i$ lies in $\Zpl$. Different choices of basis lead
to equivalent systems of congruences with the same solution set $S_g$. (Explicit choices, involving
Stirling numbers, are known, but we do not need these here.)

The following proposition is a stronger version of the congruence result 
of~\cite{strongw1}. The proof closely follows that of
~\cite[Proposition 16]{strongw1}.

\begin{proposition}\label{improvedcongruence}
Fix $n\geq 1$.
Suppose that an operation $\theta\in\mathcal{A}(BP)$ is such that its action on
homotopy $\theta_*:BP_*\to BP_*$ satisfies the following conditions. For 
each $i\geq 0$, there is some $\mu_i\in\Zpl$ such that
\begin{enumerate}
\item $\theta_*(x)\equiv \mu_i x \mod J_n$ if $x\notin J_n$, $|x|=2(p-1)i$, and 
\item $\theta_*(x)=0$ if $x\in J_n$.
\end{enumerate}
Then $(\mu_i)_{i\geq 0} \in S_g$.
\end{proposition}

\begin{proof}
Under the isomorphism
$PBP^0(\underline{BP}_0)\cong \Hom_{BP_*}(QBP_{\ast}(\underline{BP}_0), BP_*)$,
the operation $\theta$ corresponds 
to a $BP_*$-linear functional $\overline{\theta}:QBP_{\ast}(\underline{BP}_0)\to BP_*$
of degree zero.

Let $V_{\mu}: QBP_{\ast}(\underline{BP}_0)\to \Zpl$ be the composite
$\pi\overline{\theta}$ where $\pi : BP_{\ast}\rightarrow \Zpl$ is
defined to be the ring map determined by
    \begin{align*}
    v_1&\mapsto 1,\\
    v_i&\mapsto 0, \quad\text{for $i>1$}.
    \end{align*}

Thus we have a commutative diagram
    $$
    \xymatrix{ QBP_{\ast}(\underline{BP}_0) \ar[rr]^-{\overline{\theta}}
    \ar[rrd]_-{V_{\mu}} && BP_{\ast} \ar[d]^-{\pi} \\
    && \Zpl \\
    }
    $$

Recall from~\cite{bjw} that $QBP_{\ast}(\underline{BP}_0)$ is torsion-free and rationally
generated by elements of the form $v^\alpha e^{2(p-1)h} \eta_R(v^\beta)$, where
$v^\alpha\in BP_*$, $v^\beta\in BP_{2(p-1)h}$, $e\in QBP_1(\underline{BP}_1)$ is the suspension element and
$\eta_R$ is the right unit map.
By~\cite[12.4]{bjw}, the action of an operation $\theta$ on homotopy can be recovered from the
corresponding functional $\overline{\theta}$ via 
$\theta_*(v^\beta)  =\overline{\theta}(e^{2(p-1)h}\eta_R(v^\beta))$, for $v^\beta\in BP_{2(p-1)h}$.

We have
    \begin{align*}
    V_\mu(v^\alpha e^{2(p-1)h} \eta_R(v^\beta))&=\pi\overline{\theta}(v^\alpha e^{2(p-1)h} \eta_R(v^\beta))\\
        &=\pi(v^\alpha\theta_*(v^\beta))\\
        &=\begin{cases}
        \pi(v^{\alpha} (\mu_{h}v^\beta + y))&\text{for some $y\in J_n$, if $v^\beta\notin J_n$}\\
        0&\text{if $v^\beta\in J_n$}
        \end{cases}\\
        &=\begin{cases}
        \mu_h &\text{if $\alpha=(\alpha_1, 0,0,\dots)$ and $\beta=(h,0,0,\dots)$}\\
        0&\text{otherwise}.
        \end{cases}
    \end{align*}
Thus for each $x\in QBP_{\ast}(\underline{BP}_0)$,  $V_{\mu}(x)$ is some rational
linear combination of the $\mu_{i}$ and since this lies in $\Zpl$, this
gives congruences which must be satisfied by the $\mu_i$.

Consider the standard map of ring spectra $BP\to G$. This induces
a map of Hopf rings $BP_*(\underline{BP}_*)\to G_*(\underline{G}_*)$ and thus
a ring map on indecomposables $QBP_*(\underline{BP}_0)\to QG_*(\underline{G}_0)$
which we denote by $\phi$.

Now we claim that we can factorize $V_{\mu}$ as
$\pi_{\mu}\widetilde{\phi}$ where
    $
    \widetilde{\phi}:QBP_{\ast}(\underline{BP}_0)\rightarrow\textrm{Im}(\phi)
    $
is the map given by restricting the codomain of
    $
    \phi:QBP_{\ast}(\underline{BP}_0)\rightarrow QG_{\ast}(\underline{G}_0),
    $
and
    $$
    \pi_{\mu}:\textrm{Im}(\phi)\rightarrow\Zpl
    $$
is the $\Q$-linear map determined by
    $$
    \hat{u}^{a}e^{2(p-1)b}\hat{v}^{b}\mapsto\mu_{b}.
    $$

To prove the claim, it is enough to check on rational generators:
    \begin{align*}
    \pi_{\mu}\widetilde{\phi}&(v^\alpha e^{2(p-1)h} \eta_R(v^\beta))\\
    &=
    \begin{cases}
        \pi_{\mu}(\hat{u}^{\alpha_1}e^{2(p-1)h}\hat{v}^h) &\text{if $\alpha=(\alpha_1, 0,0,\dots)$ and $\beta=(h,0,0,\dots)$}\\
        0&\text{otherwise}
    \end{cases}\\
     &=\begin{cases}
        \mu_h &\text{if $\alpha=(\alpha_1, 0,0,\dots)$ and $\beta=(h,0,0,\dots)$}\\
        0&\text{otherwise}
    \end{cases}\\
    &=V_\mu(v^\alpha e^{2(p-1)h} \eta_R(v^\beta)).
    \end{align*}

Just as in~\cite[proof of Theorem~19]{strongw1}, up to some shift by a power of $\hat{u}$, each basis element of 
 $QG_0(\underline{G}_0)$, say
$\hat{f}_{n}$, for $n\geq 0$, is in
the image of the map from $QBP_{\ast}(\underline{BP}_0)$. 
So we have $x_n\in QBP_{\ast}(\underline{BP}_0)$
such that $\phi(x_n)=\hat{u}^{c_n}\hat{f}_n$, for some $c_n\in\Z$.
Then $V_{\mu}(x_n)= \pi_{\mu}\widetilde{\phi}(x_n) = \pi_{\mu}\left(\hat{u}^{c_n}\hat{f}_n\right)$.

But $V_{\mu}(x_n)\in\Zpl$ and $\pi_\mu(\hat{u}^{c_n}\hat{f}_n)\in\Zpl$ is exactly
the $n$-th congruence condition for $g$.
Hence $(\mu_i)_{i\geq 0}\in S_{g}$.
\end{proof}

Now we show how this applies to $\bpn$ operations.

\begin{proposition}\label{applytobpn}
Let $n\geq 1$ and $\phi\in \mathcal{D}(\bpn)$. Then $i_n(\phi)\in \mathcal{A}(BP)$
satisfies the hypotheses of Proposition~\ref{improvedcongruence}.
\end{proposition}

\begin{proof}
Let $\phi\in \mathcal{D}(\bpn)$, where the action of
$\phi_*$ on $\pi_{2(p-1)i}(\underline{\bpn}_0)$ is multiplication
by $\mu_i$.
Then, using $[-]$ to denote classes modulo $J_n$, for $v^\alpha\in BP_*$,
by part (2) of Lemma~\ref{actions},
    \begin{align*}
    (i_n(\phi))_*(v^\alpha)
        &=\begin{cases}
        \phi_*([v^\alpha]) \mod J_n &\text{if $v^\alpha\notin J_n$}\\
        0&\text{if $v^\alpha\in J_n$}
        \end{cases}\\
        &=\begin{cases}
        \mu_{|\!|\alpha|\!|}v^\alpha \mod J_n &\text{if $v^\alpha\notin J_n$}\\
        0&\text{if $v^\alpha\in J_n$},
        \end{cases}
    \end{align*}
where $|\!|\alpha|\!|=\frac{|\alpha|}{2(p-1)}$.
\end{proof}

Putting everything together gives the following.

\begin{theorem}\label{main}
For all $n\geq 1$,
the image of the injective ring homomorphism $\hat{\iota}_n:\mathcal{A}(g)\hookrightarrow \mathcal{A}(\bpn)$
is the centre $Z(\mathcal{A}(\bpn))$ of $\mathcal{A}(\bpn)$.
\end{theorem}

\begin{proof}
We have $\image(\hat{\iota}_n)\subseteq Z(\mathcal{A}(\bpn))= \mathcal{D}(\bpn)$, 
by Proposition~\ref{mapprops} and Theorem~\ref{centreisdiag}. Now let $\phi\in\mathcal{D}(\bpn)$,
where $\phi$ acts on $\pi_{2(p-1)i}(\underline{\bpn}_0)$ as multiplication
by $\mu_i$. By Proposition~\ref{faithfulaction}, $\phi$ is completely determined by
the sequence $(\mu_i)_{i\geq 0}$. By Proposition~\ref{applytobpn}, $i_n(\phi)$ satisfies the hypotheses
of Proposition~\ref{improvedcongruence} for the same sequence $(\mu_i)_{i\geq 0}$, and so by Proposition~\ref{improvedcongruence},
 $(\mu_i)_{i\geq 0}\in S_g$. Thus we have the following commutative diagram.
    $$
    \xymatrix{
    \mathcal{A}(g) \ar[r]^-{\cong}_-{\hat{\iota}_n} \ar[d]_-{\cong} &
    \textrm{Im}(\hat{\iota}_n) \ar@{^{(}->}[r] & Z(\mathcal{A}(\bpn))
    \ar[r]^-{=} &\mathcal{D}(\bpn)
    \ar@{^{(}->}[d]_-{} \\
    S_{g} \ar[rrr]^-{=} &&& S_{g}
    }
    $$
So the inclusions are equalities.
\end{proof}

\end{document}